\documentclass[12pt]{amsart}
\usepackage{amssymb,amsfonts,amsthm,latexsym,bm}
\usepackage{hyperref}
\usepackage[centertags]{amsmath}
\usepackage{mathrsfs}
\usepackage{t1enc}
\usepackage{graphicx}
\usepackage{enumerate}

\advance\headheight by 3pt

\newtheorem{thm}{Theorem}[section]
\newtheorem{cor}[thm]{Corollary}
\newtheorem{lem}[thm]{Lemma}

\newtheorem*{thm*}{Theorem}

\newtheorem*{subspacetheorem*}{Subspace Theorem}

\numberwithin{equation}{section}

\renewcommand{\subjclass}[1]{\thanks{\emph{2010 Mathematics Subject Classification:}~#1}}
\renewcommand{\keywords}[1]{\thanks{\emph{Keywords and Phrases:}~#1}}
\renewcommand{\date}{\thanks{\today}}

\newcommand{\xv}{{\bf x}}

\renewcommand{\AA}{\mathcal{A}}

\newcommand{\NN}{\mathcal{N}}

\newcommand{\RR}{\mathcal{R}}
\renewcommand{\SS}{\mathcal{S}}

\newcommand{\Cc}{\mathbb{C}}
\newcommand{\Rr}{\mathbb{R}}
\newcommand{\Qq}{\mathbb{Q}}

\newcommand{\Zz}{\mathbb{Z}}

\newcommand{\Aa}{\mathbb{A}}

\newcommand{\medfrac}[2]{\mbox{\large{$\textstyle{\frac{#1}{#2}}$}}}

\newcommand{\half}{\mbox{$\textstyle{\frac{1}{2}}$}}

\renewcommand{\ge}{\geq}
\renewcommand{\le}{\leq}
\newcommand{\kdots}{,\ldots ,}

\renewcommand{\gcd}{{\rm gcd}}

\def\house#1{\setbox1=\hbox{$\,#1\,$}%
\dimen1=\ht1 \advance\dimen1 by 2pt \dimen2=\dp1 \advance\dimen2
by 2pt
\setbox1=\hbox{\vrule height\dimen1 depth\dimen2\box1\vrule}%
\setbox1=\vbox{\hrule\box1}%
\advance\dimen1 by .4pt \ht1=\dimen1 \advance\dimen2 by .4pt
\dp1=\dimen2 \box1\relax}

\def\pitem{\advance\leftskip3mm\advance\linewidth-3mm}
\def\mitem{\advance\leftskip-3mm\advance\linewidth3mm}

\begingroup\makeatletter\ifx\SetFigFont\undefined
\gdef\SetFigFont#1#2#3#4#5{
  \reset@font\fontsize{#1}{#2pt}
  \fontfamily{#3}\fontseries{#4}\fontshape{#5}
  \selectfont}
\fi\endgroup

\title[Mahler's work on Diophantine equations]
{Mahler's work on Diophantine equations and subsequent developments}
\date

\author[J.-H. Evertse]{Jan-Hendrik Evertse}
\address{J.-H. Evertse \newline
         \indent Universiteit Leiden, Mathematisch Instituut, \newline
         \indent Postbus 9512, 2300 RA Leiden, The Netherlands}
\email{evertse\char'100math.leidenuniv.nl}

\author[K. Gy\H{o}ry]{K\'{a}lm\'{a}n Gy\H{o}ry}
\address{K. Gy\H{o}ry \newline
         \indent Institute of Mathematics, University of Debrecen \newline
         \indent H-4002 Debrecen, P.O. Box 400, Hungary}
\email{gyory\char'100science.unideb.hu}

\author[C.L. Stewart]{Cameron L. Stewart}
\address{C.L. Stewart \newline
          \indent Department of Pure Mathematics, 
                  University of Waterloo\newline
          \indent Waterloo, Ontario, Canada, N2L 3G1}
\email{cstewart\char'100uwaterloo.ca}

\begin{document}

\maketitle

The main body of Mahler's work on Diophantine equations
consists of his 1933 papers \cite{Mahler1933a, Mahler1933b, Mahler1933c}, 
in which he proved
a generalization of the Thue-Siegel Theorem on the approximation
of algebraic numbers by rationals, involving $P$-adic absolute values,
and applied this to get finiteness results for the number of solutions
for what became later known as Thue-Mahler equations.
He was also the first to give upper bounds for the number of solutions
of such equations. 
In fact, Mahler's extension of the Thue-Siegel Theorem made it possible
to extend various finiteness results for Diophantine equations
over the integers to $S$-integers, for any arbitrary finite set of primes 
$S$.
For instance Mahler himself \cite{Mahler1933d} extended Siegel's
finiteness theorem on integral points on elliptic curves to 
$S$-integral points.

In this paper we discuss Mahler's work on Diophantine approximation
and its applications to Diophantine equations, in particular Thue-Mahler equations, $S$-unit equations and $S$-integral points on elliptic curves,
and go into later developments concerning the number of solutions to Thue-Mahler equations and effective finiteness
results for Thue-Mahler equations. For the latter we need estimates for
$P$-adic logarithmic forms, which may be viewed as an outgrowth of Mahler's
work on the $P$-adic Gel'fond-Schneider theorem \cite{mah}. 
We also go briefly into decomposable form equations,
these are certain higher dimensional generalizations of Thue-Mahler equations.

\section{Mahler's $P$-adic generalization of the Thue-Siegel Theorem}\label{section1}

We adopt the convention that if $p/q$ denotes a rational number then
$p,q$ are integers with $\gcd (p,q)=1$ and $q>0$. 
Thue \cite{Thue1909} proved in 1909 that if $\zeta$ is any real algebraic number of degree $n$
and
$\beta >\half n +1$, then the inequality
\[
|\zeta -\medfrac{p}{q}|\leq q^{-\beta}
\]
has only finitely many solutions in rational numbers $p/q$.
This was later generalized by Siegel \cite{Siegel1921} in 1921
to all $\beta$
with
\begin{equation}\label{1.1}
\beta >\beta_0:=\min_{s=1\kdots n-1}\Big(\medfrac{n}{s+1}+s\Big).
\end{equation}
This condition is satisfied for instance if $\beta\geq 2\sqrt{n}$.
Siegel's result was further extended in the late 1940-s to $\beta >\sqrt{2n}$
by Gel'fond \cite{Gelfond1960} and Dyson \cite{Dyson1947},
and finally by Roth \cite{Roth1955} 
in 1955 to the weakest possible condition
$\beta >2$.

To state Mahler's generalization of Siegel's theorem we introduce
some notation.
For a prime number $P$,
we denote by $|\cdot |_P$ the standard $P$-adic absolute value on $\Qq$
with $|P|_P=P^{-1}$. This has a unique extension
to $\overline{\Qq_P}$. To uniformize our notation,
we write $|\cdot |_{\infty}$ for the ordinary absolute value,
and $\Qq_{\infty}$, $\overline{\Qq_{\infty}}$ for $\Rr$ and $\Cc$,
respectively. Further, we set $M_{\Qq}:=\{\infty\}\cup\{ {\rm primes}\}$.
We use the notation $|p,q|$ for $\max (|p|,|q|)$.

Now Mahler's main theorem on Diophantine 
approximation \cite[Satz 1, p. 710]{Mahler1933a} can be stated
somewhat more efficiently as follows:

\begin{thm}\label{thm1.1}
Let $S$ be a finite subset of $M_{\Qq}$,
let $f(X)\in\Zz [X]$ be an irreducible polynomial of degree $n\geq 3$
having a zero $\zeta_P\in\Qq_P$ for $P\in S$, and let $k\geq 1$
and $\beta>\beta_0$. Then the inequality
\begin{equation}\label{1.2}
\prod_{P\in S}\min \Big(1,\left|\zeta_P-\medfrac{p}{q}\right|_P\Big)\leq k\cdot |p,q|^{-\beta}
\end{equation}
has only finitely many solutions in rational numbers $p/q$.
\end{thm}

An important step in the proof is to reduce the single inequality
\eqref{1.2} to a finite number of systems of inequalities each
of which involves only one absolute value.
To this end, Mahler used a combinatorial argument which was also
an important tool in later work, e.g., quantitative versions
of the $P$-adic Subspace Theorem.

Choose $\beta_1$ with 
\[
\beta_0<\beta_1<\min (n,\beta )
\]
and define $\lambda$ by $\beta = (1+\lambda )\beta_1$.
We restrict ourselves to rational solutions $p/q$ of \eqref{1.2} with
\begin{equation}\label{1.3}
|p,q|\geq k^{1/\beta_1}.
\end{equation}
For each such solution we have
$k\cdot |p,q|^{-\beta}\leq (k\cdot |p,q|^{-\beta_1})^{1+\lambda}\leq 1$
since $k\geq 1$,
and
\[ 
\min \Big(1,\left|\zeta_P-\medfrac{p}{q}\right|_P\Big)=
(k\cdot |p,q|^{-\beta_1})^{(1+\lambda )\gamma_P}\ \ \mbox{for } P\in S,
\]
where $\gamma_P\geq 0$ for $P\in S$ and $\sum_{P\in S}\gamma_P\geq 1$.    
Take
\begin{equation}\label{1.4} 
v:=1+\left[\frac{t+1}{\lambda}\right]=1+\left[\frac{\beta_1}{\beta -\beta_1}\cdot (t+1)\right]
\end{equation}
and let $g_P:=[(1+\lambda)\gamma_P\cdot v]$ for $P\in S$.
Then $\sum_{P\in S} g_P\geq v$, hence there are integers $f_P$
with $0\leq f_P\leq g_P$ for $P\in S$ and $\sum_{P\in S} f_P=v$.
Taking 
$\Gamma_P:=f_P/v$ for $P\in S$, we get
$\Gamma_P\leq (1+\lambda )\gamma_P$ for $P\in S$ and thus,
\begin{equation}\label{1.5}
\min \Big(1,\left|\zeta_P-\medfrac{p}{q}\right|_P\Big)\leq
(k\cdot |p,q|^{-\beta_1})^{\Gamma_P}\ \ \mbox{for } P\in S. 
\end{equation}
Notice that $(\Gamma_P:\, P\in\ S )$ belongs to the finite set
$\SS$ of rational tuples
$(f_P/v:\, P\in S)$, where the $f_P$ are non-negative integers
with $\sum_{P\in S}f_P=v$. For later reference we record that
for all $(\Gamma_P:\, P\in S)\in\SS$ we have
\begin{equation}\label{1.6}
\sum_{P\in S}\Gamma_P=1,\ \ \Gamma_P\geq 0\ \mbox{for }P\in S
\end{equation}
and, by \eqref{1.4},
\begin{equation}\label{1.7}
|\SS |=\binom{v+t}{t}\leq 2^{v+t}
\leq 2^{\frac{\beta}{\beta-\beta_1}\cdot (t+1)},
\end{equation}
where here and below, we denote by $|\AA |$ the cardinality of a set $\AA$.
Thus, inequality \eqref{1.2} has been reduced to at most
$|\SS |$ systems \eqref{1.5},
hence to prove Theorem \ref{thm1.1}, 
it suffices to prove that each of these systems \eqref{1.5} has only finitely many 
solutions $p/q\in\Qq$. In fact, Mahler \cite[p. 709]{Mahler1933a}
proved the following more precise result.

\begin{thm}\label{thm2}
Let $\beta_0<\beta_1<n$ and let $\Gamma_P$ ($P\in S$) be non-negative
reals with $\sum_{P\in S}\Gamma_P=1$.
There are effectively computable numbers $C_1,C_2$,
depending only on $k$, $f$ and $\beta_1$,
and thus independent of the set $S$ and the reals $\Gamma_P$, 
such that if \eqref{1.5} has a solution $p_1/q_1\in\Qq$ with $|p_1,q_1|\geq C_1$ then for any other solution $p_2/q_2\in\Qq$ of \eqref{1.5} we have
$|p_2,q_2|\leq |p_1,q_1|^{C_2}$.
\end{thm}

\begin{proof}[Sketch of proof]
Mahler's proof is an 
adaptation of that of Siegel \cite{Siegel1921}.
We give a very brief outline.
Assume that Theorem \ref{thm2} is false and take 
two rational solutions $p_1/q_1$, $p_2/q_2$ of \eqref{1.5}
such that $|p_1,q_1|\geq C_1$ and 
$\lambda:=\medfrac{\log |p_2,q_2|}{\log |p_1,q_1|}> C_2$.  
Let $r:=[\lambda ]+1$ and choose an integer $s$ with $1\leq s\leq n-1$ such that
$\medfrac{n}{s+1}+s=\beta_0$. Let $m$ be an integer with $m>r$ and $(m+1)(s+1)>nr$. 
Following Siegel, Mahler shows that there is
a polynomial $R_m(X,Y)$  
of degree at most $m$ in $X$ and degree at most $s$ in $Y$,
with integer coefficients of not too large size,
such that $R_m(p_1/q_1,p_2/q_2)\not= 0$ and 
$\medfrac{\partial^iR_m}{\partial X^i}(\zeta ,\zeta )=0$ for $i=0\kdots r-1$ and each root $\zeta$ of $f(X)$. Then
\[
R_m(X,Y)=F_m(X,Y,\zeta )(X-\zeta )^r+G_m(X,Y,\zeta )(X-\zeta)
\]
for each root $\zeta$ of $f(X)$, where $F_m(X,Y,Z)$, $G_m(X,Y,Z)$ are polynomials
with integral coefficients. Now $A_m:=q_1^mq_2^sR_m(p_1/q_1,p_2/q_2)$
is a non-zero integer, hence $\prod_{P\in S} |A_m|_P\geq 1$.
On the other hand, using that $p_1/q_1$, $p_2/q_2$ are solutions of \eqref{1.5}
and that $\beta >\beta_0$
one can deduce good upper bounds for $|A_m|_P$ for $P\in S$, and in fact show 
that there is $m$ with $\prod_{P\in S} |A_m|_P<1$.
This leads to a contradiction.
\end{proof}

Mahler used Theorem \ref{thm2} in \cite{Mahler1933b} 
to compute an upper bound for the
number of solutions of \eqref{1.5} or \eqref{1.2}. Another important
tool is the following simple observation that solutions of \eqref{1.5}
are far apart \cite[p. 40]{Mahler1933b},
nowadays called a \emph{gap principle}.

\begin{lem}\label{lem1.3}
Let $p_1/q_1$, $p_2/q_2$ be two distinct rational solutions of \eqref{1.5},
with $|p_2,q_2|\geq |p_1,q_1|> k^{1/\beta_1}$. Then
\[
|p_2,q_2|\geq \medfrac{1}{2k}\cdot |p_1,q_1|^{\beta_1-1}.
\]
\end{lem}

\begin{proof}
Write $s(\infty )=1$ and $s(P)=0$ if $P$ is a prime.
Let $S':=S\cup\{\infty\}$ and put $\Gamma_{\infty}:=0$ if $\infty\not\in S$.
Let $P\in S'$. Then if $\Gamma_P>0$ we have
$\big|\zeta_P-\medfrac{p_i}{q_i}\big|_P\leq (k\cdot |p_i,q_i|^{-\beta})^{\Gamma_P}$
for $i=1,2$ and so
\begin{eqnarray*}
|p_1q_2-p_2q_1|_P
&\leq& (|p_1,q_1|\cdot |p_2,q_2|)^{s(P)}\left|\medfrac{p_1}{q_1}-\medfrac{p_2}{q_2}\right|_P
\\
&\leq& (2|p_1,q_1|\cdot |p_1,q_2|)^{s(P)}\max \Big(\left|\medfrac{p_1}{q_1}-\zeta_P\right|_P,
\left|\medfrac{p_2}{q_2}-\zeta_P\right|_P\Big)
\\
&\leq& (2|p_1,q_1|\cdot |p_1,q_2|)^{s(P)}(k\cdot |p_1,q_1|^{-\beta_1})^{\Gamma_P}
\end{eqnarray*}
while if $\Gamma_P=0$ we have the trivial estimate 
\[
|p_1q_2-p_2q_1|_P\leq (2|p_1,q_1|\cdot |p_1,q_2|)^{s(P)}.
\]
By taking the product over $P\in S'$, using
$\prod_{P\in S'} |p_1q_2-p_2q_1|_P\geq 1$, we obtain
\[
1\leq 2k\cdot |p_1,q_1|^{1-\beta_1}\cdot |p_2,q_2|,
\]
which implies our lemma.
\end{proof}

Assume $n=\deg f\geq 3$.
Combining Theorem \ref{thm2} with Lemma \ref{lem1.3}, one easily deduces that
the number of solutions of \eqref{1.5} is at most a quantity
depending only on $k$, $f$ and $\beta_1$. Invoking \eqref{1.7},
we arrive at the following 
(see \cite[pp. 46--47, Hilfssatz 3]{Mahler1933b}):

\begin{thm}\label{thm1.4} (i) System \eqref{1.5} has at most $C_3$ 
solutions in rationals $p/q$ with $|p,q|>k^{1/\beta_1}$, 
where $C_3$ is effectively computable and depends only on $k$, $f$ and $\beta_1$;
\\
(ii) Inequality \eqref{1.2} has at most $C_3\cdot 2^{\frac{\beta}{\beta-\beta_1}\cdot (t+1)}$ solutions in rationals $p/q$
with $|p,q|>k^{1/\beta_1}.$
\end{thm}

Theorems \ref{thm1.1} and \ref{thm1.4} have been generalized and refined in several respects.
Already Siegel \cite{Siegel1921} considered approximation of algebraic
numbers by elements from a given number field.
In his PhD-work from 1939, published much later in a journal
\cite{Parry1950}, Parry proved a common generalization of
the results of Siegel and Mahler for inequalities over a given number field
involving various archimedean and non-archimedean absolute values.
Roth \cite{Roth1955} and Ridout \cite{Ridout1957} extended the results
of Siegel and Mahler from
$\beta >\beta_0$  to $\beta >2$.
Lang \cite{Lang1960} extended this further to number fields,
thereby also improving Parry's result.

Also there has been much work on estimating the number of approximants.
Davenport and Roth \cite{DavenportRoth1955} were the first to give an upper
bound for the number of solutions of $|\zeta -p/q|\leq q^{-\beta}$ 
in rational numbers $p/q$, where $\zeta$ is a real algebraic number
and $\beta$ is any real $>2$. Their bound was improved by Mignotte
and then further by Bombieri and van der Poorten.
We give the up to now best quantitative result for the number of solutions
of a system \eqref{1.5}, which can be deduced from a more general result of
Bugeaud and Evertse \cite[Prop. A.1]{BugeaudEvertse2008} over number fields.
The height $H(P)$ of a polynomial $P$ with integer coefficients is the maximum
of the absolute values of its coefficients. 

\begin{thm}\label{thm1.5}
Let $k\geq 1$,  let 
$f(X)\in \Zz[X]$ be a polynomial of degree $n\geq 3$,  let $S$ be a finite subset of $M_{\Qq}$,
for $P\in S$ let $\zeta_P$ be a zero of $f(X)$ in $\overline{\Qq_P}$,
let $\beta_1>2$,
and let $\Gamma_P$ ($P\in S$) be non-negative reals with $\sum_{P\in S}\Gamma_P=1$. Put $\delta :=1-2/\beta_1$. 
Then the number of $p/q\in\Qq$ with 
\begin{eqnarray*}
&|p,q|\geq (4k)^{2/(\beta_1-2)}H(f),&
\\[0.2cm]
&\min \Big(1,\left|\zeta_P-\medfrac{p}{q}\right|_P\Big)\leq (k\cdot |p,q|^{-\beta_1})^{\Gamma_P}\ \mbox{for }
P\in S&
\end{eqnarray*}
is at most
$2^{30}\delta^{-3}\log 3n\cdot\log (\delta^{-1}\log 3n)$.
\end{thm}

As said, Bugeaud and Evertse proved a generalization of this result over number
fields. There are further, higher dimensional generalizations, namely,
quantitative versions of Schmidt's Subspace Theorem \cite{Schmidt1989} 
and generalizations by Schlickewei \cite{Schlickewei1992}, which allow the unknowns
to be taken from a number field and which involve both archimedean and
non-archimedean absolute values. Evertse and Schlickewei and lastly
Evertse and Ferretti obtained various
sharpenings of Schlickewei's result. We refer only to \cite{EvertseFerretti2013}, which contains the sharpest result,
as well as a historical overview of the subject.

\section{Thue-Mahler equations and $S$-unit equations}\label{section2}

Let $F(X,Y)=a_0X^n+a_1X^{n-1}Y+\cdots +a_nY^n\in\Zz [X,Y]$
be an irreducible (i.e., over $\Qq$) binary form of degree $n\geq 3$.
Thue \cite{Thue1909} proved that for every positive integer $m$,
the equation
\begin{equation}\label{2.1}
F(p,q)=m
\end{equation}
has only finitely many solutions in integers $p,q$.

Mahler considerably extended this result.
Let $\beta>\beta_0$ where as before $\beta_0=\min_{s=1\kdots n-1}\big(\medfrac{n}{s+1}+s\big)$ and let $P_1\kdots P_t$ 
be distinct prime numbers.
Put $S:=\{\infty, P_1\kdots P_t\}$.
Mahler considered the inequality
\begin{equation}\label{2.2}
\prod_{P\in S} |F(p,q)|_P\leq |p,q|^{n-\beta}
\end{equation}
to be solved in integers $p,q$ with $\gcd (p,q)=1$.
In \cite[Satz 2]{Mahler1933a} he proved that this inequality has
only finitely many solutions.
In \cite[Satz 6]{Mahler1933b}
he gave an upper bound for the number of solutions of \eqref{2.2}
which we state here in a simplified form.

\begin{thm}\label{thm2.1}
The number of solutions of \eqref{2.2} is at most $C_4^{t+1}$, where
$C_4$ is an effectively computable number that depends only on $F$ and $\beta$,
and so is independent of $P_1\kdots P_t$.
\end{thm}

\begin{proof}[Sketch of proof]
Without loss of generality we consider solutions with $q>0$.
Put $f(X):=F(X,1)$. 
Then by an elementary argument Mahler \cite[pp. 716--722]
{Mahler1933a} shows that for every 
$P\in M_{\Qq}=\{\infty\}\cup\{ {\rm primes}\}$ there is
$c_{F,P}$ with $0<c_{F,P}\leq 1$ such that
if $p,q$ are any two coprime integers with $q>0$, then 
\[
\begin{array}{ll}
|F(p,q)|_P \geq c_{F,\infty}M_{F,\infty}|p,q|^n\ \ 
&\mbox{if $P=\infty$,}
\\
|F(p,q)|_P \geq c_{F,P}M_{F,P}\ \ &\mbox{if $P$ is a prime,}
\end{array}
\]
where
\[
M_{F,P}=\left\{
\begin{array}{l}
1\ \ \mbox{if $f(X)$ has no zeros in $\Qq_P$,}
\\[0.15cm]
\min \Big(1,\left|\zeta_P-\medfrac{p}{q}\right|_P\Big)\ \ 
\mbox{otherwise,}
\end{array}\right.
\]
with $\zeta_P$ the zero of $f$ in $\Qq_P$ $P$-adically closest to $p/q$.
Further, $c_{F,P}=1$ for all but finitely many $P$.
Hence $c_F:=\prod_{P\in M_{\Qq}} c_{F,P}$ is positive. 
Now let $S'$ be the subset of $P\in S$ such that $f(X)$ has roots in $\Qq_P$.
Let $(p,q)$ be a solution of \eqref{1.2} and for $P\in S'$, 
let $\zeta_P$ be as above. Then
\[
|p,q|^{n-\beta}\geq c_F\cdot |p,q|^n\prod_{P\in S'}\min\Big(1,\left|\zeta_P-\medfrac{p}{q}\right|_P\Big),
\]
that is, $(p,q)$ satisfies \eqref{1.2} with $k=c_F^{-1}$ and $S'$ in place
of $S$ for certain roots $\zeta_P\in \Qq_P$ of $f(X)$.
Now Theorem \ref{thm2.1} follows from Theorem \ref{thm1.4},
taking into consideration that for every $P\in S'$ there are at most
$n$ possibilities for $\zeta_P$.
\end{proof}

We state some immediate consequences. We keep the assumption
that $F$ is a binary form with integer coefficients that is irreducible
over $\Qq$.
First note that if in \eqref{2.2} we take $\beta =n$,
we get $\prod_{P\in S} |F(p,q)|_P=1$ and thus,
\begin{equation}\label{2.3}
|F(p,q)|=P_1^{z_1}\cdots P_t^{z_t}
\end{equation}
where $z_1\kdots z_t$ are non-negative integers. This equation is nowadays
called the \emph{Thue-Mahler} equation.
Then from Theorem \ref{thm2.1} we immediately deduce
(cf. \cite[Folgerung 1, p. 52]{Mahler1933b}): 

\begin{cor}\label{cor2.2}
Eq. \eqref{2.3} has at most $C_5^{t+1}$ solutions in integers
$p,q$, $z_1\kdots z_t$ with $\gcd (p,q)=1$, where
$C_5$ is effectively computable and depends only on $F$.
\end{cor}

Another quick consequence gives an upper bound for the number of
solutions of \eqref{2.1}. Let $P_1\kdots P_t$ denote the primes
dividing $m$. 
Let $(p,q)$ be a solution of \eqref{2.1} and let
$(p',q')=(p/d,q/d)$
where $d=\gcd (p,q)$. Then $|F(p',q')|$ is composed of 
primes from $P_1\kdots P_t$. Further, given $(p',q')$ there is at most one
positive integer $d$ such that $(p,q):=(dp',dq')$ satisfies \eqref{2.1}.
Thus, we obtain the following quantitative version of Thue's Theorem:

\begin{cor}\label{cor2.3}
Eq. \eqref{2.1} has at most $C_5^{\omega (m)+1}$ solutions
in integers $p,q$, where
$\omega (m)$ is the number of prime divisors of $m$.
\end{cor}

We give another result of Mahler, which for later purposes we have reformulated
in more modern language. 
Let $S:=\{ P_1\kdots P_t\}$ be a set of primes.
An \emph{$S$-unit} is a 
rational number of the shape $\pm P_1^{z_1}\cdots P_t^{z_t}$ with 
$z_1\kdots z_t\in\Zz$. 
Consider the so-called \emph{$S$-unit equation}
\begin{equation}\label{2.4}
x+y=1\ \ \mbox{in $S$-units } x,y.
\end{equation}
Given such a solution, we may write $x=A/C$, $y=B/C$ where $A,B,C$
are integers with $\gcd (A,B,C)=1$, all composed of primes from $S$.
Subsequently, we can write $A=ap^5$, $B=bq^5$, 
with $a,b$
from finite sets independent of $A,B$ and $p,q$ coprime integers.
Thus, \eqref{2.4} becomes $|ap^5+bq^5|=P_1^{z_1}\cdots P_t^{z_t}$
with $z_1\kdots z_t$ non-negative integers. 
If $-a/b$ is not the fifth power 
of a rational number then the binary form
$aX^5+bY^5$ is irreducible, in which case we can directly apply
Corollary \ref{cor2.2}. Otherwise, there is a factorization 
$aX^5+bY^5=(a'X+b'Y)G(X,Y)$ where $a',b'$ are integers
and $G(X,Y)$ is an irreducible binary form of degree $4$ with integral coefficients,
and then \eqref{2.4} leads to an equation
$|G(p,q)|=P_1^{z_1}\cdots P_t^{z_t}$
to which Corollary \ref{cor2.2} can be applied. This leads to:

\begin{cor}\label{cor2.4}
Equation \eqref{2.4} has only finitely many solutions.
\end{cor}

With Mahler's results one can compute an upper bound for the number of
solutions of \eqref{2.4}, but this does not lead to anything interesting.

In 1961, Lewis and Mahler \cite{LewisMahler1961} obtained
explicit versions of Corollaries \ref{cor2.2} and \ref{cor2.4}.
In this paper they proved that if $F(X,Y)$ is a not necessarily
irreducible binary form with integer
coefficients of degree $n\geq 3$ with non-zero discriminant and
with $F(1,0)F(0,1)\not= 0$, then \eqref{2.3} has at most
\begin{equation}\label{2.5}
(c_1nH(F))^{\sqrt{n}}+(c_2n)^{t+1}
\end{equation}
solutions, where $c_1,c_2$ are absolute constants. By means of the 
argument described above, they derived an upper bound for the number of
solutions of \eqref{2.4} depending on the primes in the set $S$,
and they posed as an open problem to obtain an upper bound
depending only on the cardinality of $S$. However, the condition
$F(1,0)F(0,1)\not= 0$ is not necessary in the result of Lewis and Mahler.
Indeed, there are integers $u,v\in\{ 0\kdots n-1\}$ such that $F(1,u)\not= 0$
and $F(v,uv+1)\not= 0$. Put $G(X,Y):=F(X+vY,uX+(uv+1)Y)$.
Then $G(1,0)G(0,1)\not=0$, and the number of solutions of \eqref{2.3}
does not change if we replace $F$ by $G$, since $G$ is obtained from $F$
by means of a transformation from ${\rm SL}_2(\Zz )$.
So we can apply the result of Lewis
and Mahler with $F(X,Y)=XY(X+Y)$
and deduce at once that \eqref{2.4} has at most $c_3^{t+1}$ solutions,
where $c_3$ is an absolute constant.  
Erd\H{o}s, Stewart and Tijdeman showed in 1998 \cite{EST} that this estimate cannot be improved that much. Let $\epsilon$ be a positive real number. They proved that if $t$ is sufficiently large in terms of $\epsilon$ then there exist a set of primes $S=\{ P_1\kdots P_t\}$ such that \eqref{2.4} 
has at least $\exp ((4-\epsilon)(t/\log t)^{1/2}) $ solutions. 
In 2007 Konyagin and Soundararajan \cite{KS} improved the lower bound to $\exp(t^{2-\sqrt2 - \epsilon})$.

After the result of Lewis and Mahler it remained as an open problem
whether \eqref{2.5} can be replaced by a bound depending
only on $n=\deg F$ and $t$, so independent of the height of $F$.
In his PhD-thesis \cite{Evertse1983} Evertse established such a bound, 
though with a much worse dependence on $n$ than \eqref{2.5}.
Independently of Evertse, Mahler \cite{Mahler1984}
took up again the study of the number of solutions of Thue equations \eqref{2.1}.
He proved that if $F(X,Y)\in\Zz [X,Y]$ is an irreducible binary form
of degree $n\geq 3$ and $|m|\geq (450n^4H(F)^4)^{n/(n-2)}$, then the Thue equation
\begin{equation*}\tag{\ref{2.1}}
F(p,q)=m
\end{equation*}
has at most $64n^{\omega (m)+1}$ solutions $p,q\in\Zz$ with $\gcd (p,m)=\gcd (q,m)=1$. Some years later, Bombieri and Schmidt \cite{BombieriSchmidt1986}
proved, without any condition on $m$, 
that \eqref{2.1} has at most $c_0n^{\omega (m)+1}$
solutions $p,q\in\Zz$ with $\gcd (p,q)=1$, where $c_0$ is an absolute constant. Let $g$ be a divisor of $m$, coprime with the discriminant of $F$, with $g \geq |m|^{\frac{2.5}{n}}$. Stewart \cite{ST1} showed in 1991 that \eqref{2.1} has at most $4200n^{\omega (g)+1}$
solutions in pairs of coprime integers $(p,q)$. Erd\H{o}s and Mahler \cite{ErdosMahler1938} were the first to estimate the number of solutions of \eqref{2.1} in terms of a divisor $g$ of $m$.
Bombieri improved Evertse's bound for the number of solutions of the Thue-Mahler
equation \eqref{2.3} in terms of the degree $n$ of $F$, and this was subsequently 
improved by Evertse by a different method. 
We mention only Evertse's result \cite{Evertse1997},
which asserts that if $F(X,Y)$ is an irreducible
binary form of degree $n\geq 3$, then \eqref{2.3} has at most
\[
2\times (10^5n)^{t+1}
\]
solutions. In fact, he proved a generalization
of this for Thue-Mahler equations over number fields. Evertse's result implies
that \eqref{2.1} has at most $2\times (10^5n)^{\omega (m)+1}$ solutions $p,q\in\Zz$,
without the requirement $\gcd (p,q)=1$.

Instead of \eqref{2.4} one may consider the weighted $S$-unit equation
\begin{equation}\label{2.6a}
ax+by=1\ \ \mbox{in $S$-units } x,y,
\end{equation}
where again $S=\{ P_1\kdots P_t\}$ 
with $P_1\kdots P_t$ distinct primes, and where $a,b$ are non-zero rationals.
Evertse \cite{Evertse1984} obtained a uniform upper bound for the number of solutions, independent of $a,b$ and the primes in $S$, i.e., $7^{2t+4}$.
Again, Evertse proved a more general result over number fields.
A substantial generalization was obtained
by Beukers and Schlickewei \cite{BeukersSchlickewei1996}. From their result
it follows that if $a,b$ are any two complex numbers
and $\Gamma$ is any multiplicative subgroup
of $\Cc^*$ of finite rank $r$ (i.e., $r$ is the maximal number of multiplicatively independent elements 
that can be chosen from $\Gamma$),
then the equation
\[
ax+by=1\ \ \mbox{in } x,y\in\Gamma
\]
has at most $2^{16(r+1)}$ solutions.

Let $F(X,Y)\in\Zz [X,Y]$ be a binary form of degree $n\geq 3$
which is irreducible over $\Qq$,
$Z\geq 1$ a real,
and $P_1\kdots P_t$ distinct primes.
Put $S:=\{\infty , P_1\kdots P_t\}$.
Denote by $A_F(Z)$ the number of solutions of the inequality
\[
|F(p,q)|\leq Z\ \ \mbox{in } (p,q)\in\Zz^2.
\]
More generally, denote by $A_{F,S}(Z)$ the number of solutions of
\begin{equation}\label{2.6}
\prod_{P\in S} |F(p,q)|_P\leq Z\ \ \mbox{in $(p,q)\in\Zz^2$ with  
$\gcd (p,q,P_1\cdots P_t)=1$.}
\end{equation}
For instance, taking $Z=1$ we get the number of pairs $(p,q)$
such that $|F(p,q)|$ is composed of primes from $\{ P_1\kdots P_t\}$
and $\gcd (p,q,P_1\kdots P_t)$ $=1$.

By a minor modification of the proof of Theorem \ref{thm2.1} one
can show that $A_{F,S}(Z)$ is finite for every $Z$.
Based on unpublished work of Siegel on $A_F(Z)$,
Mahler \cite[pp. 93--94]{Mahler1933c} derived an asymptotic formula
for $A_{F,S}(Z)$. 
His result is as follows.

\begin{thm}\label{thm2.5}
There is $\sigma_{F,S}>0$ such that
\begin{eqnarray*}
&&A_{F,S}(Z)=\sigma_{F,S}Z^{2/n}\, +\, O(Z^{1/n})\ \mbox{as } Z\to\infty
\ \ \mbox{if $t_0=0$},
\\
&&A_{F,S}(Z)=\sigma_{F,S}Z^{2/n}\, +\, O(Z^{1/(n-1)}(\log Z)^{t_0-1})\
\mbox{as } k\to\infty\ \ 
\mbox{if $t_0>0$},
\end{eqnarray*}
where $t_0$ is the number of $P\in S$ such that $F(X,1)$
has a zero in $\Qq_P$, and where the implied constants depend on $F$ and $S$.
\end{thm}

\noindent
We mention that with Mahler's proof
the implied constants cannot be computed effectively.

In the special case $S=\{\infty\}$, i.e., if there are no primes,
Theorem \ref{thm2.5} 
asserts that there is $\sigma_F>0$ such that $A_F(Z)=\sigma_FZ^{2/n}+O(Z^{1/(n-1)})$
if $F(X,1)$ has a real root, and
$A_F(Z)=\sigma_FZ^{2/n}+O(Z^{1/n})$ if $F(X,1)$ has no real root.
Here, $\sigma_F$ is the area of the set of $(x,y)\in\Rr^2$
with $|F(x,y)|\leq 1$ and $\sigma_FZ^{2/n}$ is the area of
$(x,y)\in\Rr^2$ with $|F(x,y)|\leq Z$.

For arbitrary $S$,
Mahler expressed $\sigma_{F,S}$ as a product of local factors
$\prod_{P\in S} \sigma_{F,P}$.
Another formulation of $\sigma_{F,S}$ is as follows.
Let $\mu^r=\mu_{\infty}^r$ be the Lebesgue measure
on $\Rr^r$, normalized such that $\mu_{\infty}^r([0,1]^r)=1$,
for a prime $P$ let $\mu_P^r$ be the Haar measure on $\Qq_P^r$
normalized such that $\mu_P^r(\Zz_P^r)=1$ and subsequently
let $\mu_S^r$ be the product measure on 
$\Aa_S^r:=\prod_{P\in S}\Qq_P^r$. Write elements of this product as
$(\xv_P)_{P\in S}$ where $\xv_P\in\Qq_P^r$ and for $\xv_P=(x_1\kdots x_r)\in\Qq_P^r$, write $|\xv_P|_P:=\max_i|x_i|_P$. Define the sets
\begin{eqnarray*}
&&\SS_{F,S}(Z):=\Big\{ (\xv_P)_{P\in S}\in\Aa_S^2:\,\prod_{P\in S}|F(\xv_P)|_P\leq Z,
\\[-0.2cm]
&&\hspace*{5.5cm}
|\xv_P|_P=1\ \mbox{for } P\in S\setminus\{\infty\}\Big\}
\end{eqnarray*}
and $\SS_{F,S}:=\SS_{F,S}(1)$.
Then 
$\sigma_{F,S}=\mu_S^2(\SS_{F,S})$ and $\sigma_{F,S}Z^{2/n}=
\mu_S^2(\SS_{F,S}(Z))$ for $Z>0$. If we identify $\xv\in\Zz^2$ with 
$(\xv)_{P\in S}\in\Aa_S^2$, then $A_{F,S}(Z)$ counts the number of
lattice points in $\SS_{F,S}(Z)$, and Theorem \ref{thm2.5} 
states in a more precise form that this number is approximately
the measure of $\SS_{F,S}(Z)$.

\begin{proof}[Sketch of proof of Theorem \ref{thm2.5}]
We give a very brief outline of Mahler's lengthy proof.
Mahler divides the solutions $(p,q)$ of
\eqref{2.6}
into \emph{large} solutions,
i.e., with $|p,q|\geq (4Z)^{1/(n-2)}$, \emph{medium} solutions,
i.e., with $Z^{1/(n-1)}\leq |p,q|<(4Z)^{1/(n-2)}$,
and \emph{small} solutions, i.e., with $|p,q|<Z^{1/(n-1)}$.
Mahler estimates the number of large numbers 
using the approximation techniques discussed in the previous section,
and the number of medium solutions
by means of an elementary argument
based on congruences and continued fractions.
As it turns out, both the numbers of 
large and medium solutions go into the error term.
Then Mahler estimates the number of small solutions by 
combining congruence results for binary forms
with elementary estimates for the difference between 
the number of lattice points
in a bounded two-dimensional region and the area of that region.
\end{proof}

Since we know that for the number of solutions of the Thue-Mahler
equation we have an upper bound
independent of the coefficients of the involved binary form $F$,
one may wonder whether $A_{F,S}(Z)$ can be estimated from above
independently of the coefficients of $F$. More precisely,
one may ask whether $\sigma_{F,S}$ and the implied constant of the
error term can be estimated from above independently of $F$.

Some progress has been made on these problems.
Bean \cite{Bean1994} showed that if $F(X,Y)\in\Rr [X,Y]$ is a binary form
of degree $n\geq 3$ with discriminant $D(F)\not= 0$,
then $\sigma_F\leq 16|D(F)|^{-1/n(n-1)}$.
Thunder \cite{Thunder1998} proved that if $F(X,Y)\in\Zz [X,Y]$
is a cubic binary form which is irreducible over $\Qq$, then
\[
|A_F(Z)-Z^{2/3}\sigma_F|< 9
+\,\frac{2000k^{1/3}}{|D(F)|^{1/2}}\,+ 3156Z^{1/3}.
\]
Later, Thunder proved some general results for inequalities 
involving decomposable forms. Specialized to binary forms,
these results imply that if $F(X,Y)\in\Zz [X,Y]$ is a binary form
of degree $n$ with non-zero discriminant, having no linear
factor over $\Qq$, then for all $Z\geq 1$ we have
\begin{eqnarray*}
&&A_F(Z)\leq c_1(n)Z^{2/n}\ \ \mbox{if } n\geq 3 \ \mbox{\cite{Thunder2001}},
\\
&&|A_F(Z)-\sigma_FZ^{2/n}|\leq c_2(n)Z^{2/(n+1)}\ \ 
\mbox{if $n\geq 3$, $n$ odd}\ \mbox{\cite{Thunder2005}},
\end{eqnarray*}
where $c_1(n)$, $c_2(n)$ are effectively computable and depend on $n$ only. 
An estimate
with error term depending only on $n$ and $Z$ has not been deduced yet
for even $n$. 

In his PhD-thesis \cite{Liu2015}, Liu generalized the work
of Thunder on decomposable forms
to the $p$-adic case. We mention some consequences for binary forms.
Let $P_1\kdots P_t$ be distinct primes, $S=\{\infty ,P_1\kdots P_t\}$
and $F(X,Y)\in\Zz [X,Y]$ a binary form of degree $n$
with non-zero discriminant and without linear factor over $\Qq$.
Then 
\begin{eqnarray*}
&&\sigma_{F,S}\leq c_1(n,S),\ \ A_{F,S}(Z)\leq c_2(n,S)Z^{2/n}\ \ 
\mbox{if } n\geq 3,
\\
&&
|A_{F,S}-\sigma_{F,S}Z^{2/n}|\leq c_3(n,S) Z^{2/(n+1)}(1+\log Z)^{2n(t+1)}\ \ 
\mbox{if $n\geq 3$, $n$ odd,}
\end{eqnarray*}
where the $c_i(n,S)$ are
effectively computable and depend only on $n$ and $S$.
The constants $c_i(n,S)$ that arise from Liu's proof
depend on the sizes of the primes in $S$.
Another open problem is, whether the $c_i(n,S)$ can be replaced
by numbers depending only on $n$ and the cardinality of $S$.
 
 Let $F$ be a binary form with integer coefficients, non-zero discriminant  and degree $n$ with $n \geq 3$. For any positive number $Z$ let $\RR_F(Z)$ denote the set of non-zero integers $h$ with $|h| \leq Z$ for which there exist integers $p$ and $p$ with $F(p,q) = h$. 
Denote the cardinality of a set  $\AA$ by $|\AA|$ and put 
$R_F(Z) := |\RR_F(Z)|$. 
In 1938 Erd\H{o}s and Mahler \cite{ErdosMahler1938} proved that if $F$ is irreducible over ${\mathbb{Q}}$  then there exist positive numbers $c_1$ and $c_2$, which depend on $F$, such that 
\[
R_F(Z) > c_1 Z^{\frac{2}{n}}
\]
for $Z > c_2$. \
 In 1967 Hooley \cite{Hoo1} determined the asymptotic growth rate of $R_F(Z)$ when $F$ is an irreducible binary cubic form with discriminant which is not a square. He proved that
\begin{equation} \label{Hooley result 1} 
R_F(Z) = \sigma_F Z^{\frac{2}{3}} + O \left(Z^{\frac{2}{3}} (\log \log Z)^{-\frac{1}{600}}\right). 
\end{equation}
In 2000 Hooley \cite{Hoo2} obtained an asymptotic estimate for $R_F(Z)$ in the case when the discriminant is a perfect square.   Hooley \cite{Hoo1-2} also obtained such an estimate  for quartic forms of the shape    
\[F(X,Y) = aX^4 + 2bX^2 Y^2 + cY^4\]  
and for forms which are the product of linear factors over the rationals \cite{Hooley}. Several authors, including Bennett, Dummigan, and Wooley \cite{BDW}, Browning \cite{Br2}, Greaves \cite{Grea}, Heath-Brown \cite{HB3}, Hooley \cite{Hool1},  Skinner and Wooley \cite{SWoo} and Wooley \cite{Woo}, obtained an asymptotic estimate for $R_F(Z)$ when $F$ is a binomial form. Stewart and Xiao \cite{SX} have recently proved that if $F$ is a binary form with integer coefficients, non-zero discriminant  and degree $n$ with $n \geq 3$ then there is a positive number $C_F$ such that $R_F(Z)$ is asymptotic to $C_F Z^{\frac{2}{n}}$. In the case that the form $F$ is irreducible over the rationals a key ingredient in the proof is Theorem \ref{thm2.5}. When $F$ is reducible one appeals to a special case of a result of Thunder \cite{Thunder2001}.

Let $k$ be an integer with $k \geq 2$. An integer is said to be $k$-free if it is not divisible by the $k$-th power of a prime number. For any positive number $Z$ let $\RR_{F,k}(Z)$ denote the set of $k$-free integers $h$ with $|h| \leq Z$ for which there exist integers $p$ and $q$ such that $F(p,q) = h$ and put 
$R_{F,k}(Z) := |\RR_{F,k}(Z)|$. Gouv\^{e}a and Mazur \cite{GM} in 1991 proved that if there is no prime $P$ such that $P^2$ divides $F(a,b)$ for all pairs of integers $(a,b)$, if all the irreducible factors of $F$ over $\Qq$ have degree at most $3$ and if $\epsilon$ is a positive real number then there are positive numbers $C_1$ and $C_2$, which depend on $\epsilon$ and $F$, such that if $Z$ exceeds $C_1$ then
\begin{equation} \label{basic inequality} R_{F,2}(Z) > C_2 Z^{\frac{2}{n}- \epsilon}.\end{equation}
This was subsequently extended by Stewart and Top in \cite{ST2}. Let $r$ be the largest degree of an irreducible factor of $F$ over $\Qq$. Let $k$ be an integer with $k \geq 2$ and suppose that there is no prime $P$ such that $P^k$ divides $F(a,b)$ for all integer pairs $(a,b)$. They showed, by utilizing an argument of Greaves \cite{Gre} and the result of Erd\H{o}s and Mahler \cite{ErdosMahler1938},
that if $k$ is at least $(r-1)/2$ or $k = 2$ and $r = 6$ then there are positive numbers $C_3$ and $C_4$, which depend on $k$ and $F$, such that if $Z$ exceeds $C_3$ then
\begin{equation} \label{basic inequality 2} R_{F,k}(Z) > C_4 Z^{\frac{2}{n}}. \end{equation}
The estimates of Gouv\^{e}a and Mazur and of Stewart and Top were used to count the number of twists of an elliptic curve defined over the rationals for which the rank of the group of rational points on the curve is at least $2$.

 Let $F$ be a binary form with integer coefficients, non-zero discriminant and degree $d$ with $d$ at least $3$ and let $r$ denote the largest degree of an irreducible factor of $F$ over the rationals. Let $k$ be an integer with $k \geq 2$ and suppose 
again that there is no prime $P$ such that $P^k$ divides $F(a,b)$ for all pairs of integers $(a,b)$. Stewart and Xiao \cite{SX1}  proved that there is a positive number $C_{F,k}$ such that   $R_{F,k}(Z)$ is asymptotic to $C_{F,k} Z^{\frac{2}{n}}$ 
provided that $k$ exceeds $\medfrac{7r}{18}$ or $(k,r)$ is $(2,6)$ or $(3,8)$. For a positive number $Z$ we put
\[
\NN_{F,k}(Z) := \{(p,q) \in \Zz^2 : F(p,q) \text{ is }k\text{-free and } 1 \leq |F(p,q)| \leq Z\}
\]
and
\[
N_{F,k}(Z) := |\NN_{F,k}(Z)|.
\]
For each positive integer $m$ we put
\[
\rho_F(m) := |\{(i,j) \in \{0, \cdots, m-1\}^2 : F(i,j) \equiv 0 \pmod{m} \}|
\]
and
\[\lambda_{F,k} := \prod_P \left(1 - \frac{\rho_F(P^k)}{P^{2k}} \right),\]
where the product is taken over the primes $P$. A first step in the proof of the estimate for $R_{F,k}(Z)$ is to estimate $N_{F,k}(Z)$ provided that $k$ exceeds $  \medfrac{7r}{18}$ or $(k,r)$ is $(2,6)$ or $(3,8)$.  Stewart and Xiao \cite{SX1} proved that under this assumption  
\begin{equation} \label{est11}N_{F,k}(Z) \sim \lambda_{F,k} \sigma_{F}Z^{\frac{2}{n}}  \end{equation}
which extends Mahler's result \cite{Mahler1933c}.

\section{Cubic Thue equations with many solutions}\label{section3}
Chowla \cite{SDC} was the first to show that there are cubic Thue equations with many solutions. He proved in 1933 that there is a positive number $c_0$ such that if $k$ is a non-zero integer then the number of solutions of $p^3-kq^3=m$ in integers $p$ and $q$ is at least $c_0\log\log m$ for infinitely many positive integers $m.$ This was refined by Mahler \cite{Mahler1935} in 1935. Let $F(X,Y)$ be a cubic binary form with integer coefficients and non-zero discriminant.  Let $m$ be a non-zero integer and consider the equation
\begin{equation}\label{est111}
F(p,q)=m,
\end{equation}
in integers $p$ and $q$. Mahler proved that there is a positive number $c_1$, which depends on $F$, such that for infinitely many positive integers $m$ equation \eqref{est111} has at least
\begin{equation}\label{est222}
c_1(\log m)^{1/4} 
\end{equation}
solutions. In 1983, Silverman \cite{JHS} proved that the exponent of $1/4$ in \eqref{est222} can be improved to 
$1/3$. Silverman streamlined the approach of Mahler by introducing the theory of height functions on elliptic curves into the argument. Chowla, Mahler and Silverman  obtained their results by viewing \eqref{est111}, when it has a rational point, as defining an elliptic curve $E$ and then by constructing, from rational points on $E$, integers $m'$ for which $F(p,q)=m'$ has many solutions in integers $p$ and $q$.  The solutions 
$(p,q)$ constructed by this method have very large common factors. By showing that it is always possible to find a twist of $E$ for which the rank of the group of rational points is at least $2$ Stewart \cite{CLS} was able to make a further improvement on Mahler's result. He showed that \eqref{est222} holds with $1/4$ replaced by $1/2$. In addition Stewart \cite{CLS1}, utilizing some elliptic curves whose group of rational points has rank $12$ discovered by Quer \cite{JQ}, showed that there are infinitely many cubic binary forms with integer coefficients, content 1 and non-zero discriminant which are inequivalent under the action of $GL(2,\mathbb{Z})$ and for which the estimate  \eqref{est222} applies with $1/4$ replaced by $6/7$.

\section{$S$-integral points on elliptic curves}\label{sectionx}
In his celebrated paper \cite{Siegel1929}, Siegel proved that 
non-singular affine plane curves of genus at least $1$ over $\Qq$
have only finitely many points in $\Zz^2$. In fact, he proved a more
general result for curves defined over a number field.
Mahler \cite{Mahler1933d} proved a generalization to $S$-integers
in a special case, namely for curves of genus $1$ over $\Qq$.
We state his result. In what follows, $S=\{ P_1\kdots P_t\}$ is
a finite set primes and $\Zz_S$ is the
corresponding ring of $S$-integers, i.e., the ring of rational numbers
whose denominators do not contain any prime factor different from
$P_1\kdots P_t$.

\begin{thm}\label{thmx}
Let $f\in\Qq [X,Y]$ be such that $f(x,y)=0$ 
defines a non-singular affine algebraic curve of genus $1$.
Then there are only finitely many $x,y\in\Zz_S$ with $f(x,y)=0$.
\end{thm}

In his proof, Mahler closely follows Siegel and uses his own Theorem \ref{thm1.1} instead of Siegel's approximation theorem.

\begin{proof}[Sketch of proof]
Mahler proves in fact an essentially equivalent result, invariant under birational transformations over $\Qq$, which in more modern form may be stated as follows:
\\[0.2cm] 
\emph{let $E$ be a non-singular projective curve of genus $1$ over $\Qq$ and
$g\in\Qq (E)$ a non-constant rational function on $E$ over $\Qq$;
then the set of ${\bf p}\in E(\Qq )$ with $g({\bf p})\in\Zz_S$ is finite.} 
\\[0.2cm]
Assume on the contrary that this set is infinite.
After a birational transformation over $\Qq$ we may assume that $E$ is an elliptic
curve with Weierstrass equation
$y^2=x^3+Ax+B$ with $A,B\in\Qq$ and point $O$ at infinity.
Fix an integer $n$ which is later chosen to be sufficiently large. 
By the weak Mordell-Weil theorem, 
the quotient group $E(\Qq )/nE(\Qq )$ is finite. Hence there is 
${\bf p}_0\in E(\Qq )$ such that the set $A_n$ of ${\bf p}\in E(\Qq )$
with $g({\bf p}_0+n{\bf p})\in\Zz_S$ is infinite.
Suppose that $g$ has degree $r$ and let $h$ be the maximum of the orders
of its poles.
Then $g_n:\, {\bf p}\mapsto g({\bf p}_0+n{\bf p})$ is a rational function on $E$
defined over $\Qq$ of degree $n^2r$, 
whose poles still have order at most $h$.
Let $\varphi :\, (x,y)\mapsto\medfrac{ax+b}{cx+d}$ with $a,b,c,d\in\Qq$
with $ad-bc\not=0$ be such that $\varphi ({\bf q})\not=\infty$
for every pole ${\bf q}$ of $g$ on $E$.
By a straightforward generalization of Siegel's arguments,
Mahler shows that there are $P\in S':=S\cup\{\infty\}$,
a pole ${\bf q}\in E(\overline{\Qq_P})$ of $g_n$ and infinitely many 
${\bf p}\in A_n$  
such that
\[
\left|\varphi ({\bf q})-\varphi ({\bf p})\right|_P
\leq H(\varphi ({\bf p}))^{-n^2r/6|S'|h},
\]
where $H(p/q):=\max (|p|,|q|)$ for $p,q\in\Zz$ with $q>0$ and $\gcd (p,q)=1$.
The number $\varphi ({\bf q})$ is algebraic of degree at most $n^2r$ since 
${\bf q}$ is a pole of $g_n$ and $g_n$ has degree at most $n^2 r$.
If we choose $n$ large enough so that $n^2r/6|S'|h> 2\sqrt{n^2r}=2n\sqrt{r}$,
we obtain a contradiction with Theorem \ref{thm1.1}. 
\end{proof}

Siegel treated the case of curves of genus larger than $1$ by embedding
such a curve in its Jacobian, applying the Mordell-Weil theorem to this
Jacobian, and using a simultaneous Diophantine approximation argument.
After Roth's approximation theorem and height theory became available, 
Siegel's method of proof could be greatly simplified. Lang \cite{Lang1960} worked this
out and obtained a version of Siegel's theorem for curves
of arbitrary genus $g\geq 1$ over a number field $K$,
implying that for every finite set of places $S$ of $K$,
such a curve has only finitely $S$-integral points. For curves over $K$
of genus $>1$ this was of course
superseded by Faltings \cite{Faltings1983}, who proved that such curves
have only finitely many $K$-rational points. 
The booklet \cite{Fuchs-Zannier2014}
contains a translation into English by Fuchs of Siegel's paper \cite{Siegel1929},
and a paper by Fuchs and Zannier giving an overview of several methods of proof
of Siegel's theorem, including one by Corvaja and Zannier, which uses
Schmidt's Subspace Theorem instead of the Mordell-Weil theorem. 

\section{Effective finiteness results for Thue equations, Thue--Mahler equations and unit equations}\label{section4}

Consider again the Thue equation
\begin{align}\label{eq41}
F(p,q)=m\quad\textrm{in integers $p,q$}\tag{2.1}
\end{align}
and the Thue--Mahler equation
\begin{align}\label{eq42}
|F(p,q)|=P_1^{z_1}\cdots P_t^{z_t}\quad&\textrm{in integers $p,q,z_1,\ldots,z_t$}\tag{2.3}\\&\textrm{with $\gcd (p,q)=1$},\notag
\end{align}
where $F(X,Y)$ is an irreducible binary form of degree $\ge 3$ with coefficients in $\mathbb{Z},m$ a non-zero integer and $P_1,\ldots,P_t$ a set of $t\ge 0$ distinct primes. The proofs of the finiteness theorems of Thue \cite{Thue1909} and Mahler \cite{Mahler1933a,Mahler1933b} concerning equations \eqref{eq41} and \eqref{eq42}, respectively, were ineffective, i.e. did not provide any algorithm for determining the solutions of these equations.

The first effective proof for Thue theorem was given by Baker \cite{bak4} and subsequently, for Mahler's theorem, by Coates \cite{coa1,coa2}. They obtained explicit upper bounds for $|p,q|$, the maximum of the absolute values of the solutions $p,q$.

Baker's proof is based on his effective lower bounds for linear forms in the logarithms of algebraic numbers. Gel'fond \cite{gel1} and Schneider \cite{schn} proved independently of each other that if $\alpha$ and $\beta$ are algebraic numbers such that $\alpha\ne 0,1$ and $\beta$ is not rational, then $\alpha^\beta$ is transcendental. An equivalent formulation of this theorem is that if $\alpha_1,\alpha_2$ are non-zero algebraic numbers such that $\log \alpha_1$ and $\log\alpha_2$ are linearly independent over $\mathbb{Q}$, then they are linearly independent over $\overline{\mathbb{Q}}$. Further, Gel'fond \cite{gel3} gave a non-trivial effective lower bound for $|\beta_1\log\alpha_1+\beta_2\log\alpha_2|$, where $\beta_1,\beta_2$ denote algebraic numbers, not both $0$, and $\alpha_1,\alpha_2$ denote algebraic numbers different from $0$ and $1$ such that $\log\alpha_2/\log\alpha_2$ is not rational. Baker \cite{bak1,bak2,bak3} generalized the Gel'fond--Schneider theorem to arbitrary many logarithms and provided non-trivial effective lower bounds for $|\beta_1\log\alpha_1+\cdots +\beta_k\log\alpha_k|$, where $\alpha_1,\ldots,\alpha_k$ are non-zero algebraic numbers such that $\log\alpha_1,\ldots,\log\alpha_k$ are linearly independent over $\mathbb{Q}$ and $\beta_1,\ldots,\beta_k$ are algebraic numbers, not all zero.

Baker's effective estimates for logarithmic forms led to significant applications in Diophantine equations and other parts of number theory. Baker \cite{bak4} showed that $|p,q|\le C$ for every solution $p,q$ of \eqref{eq41}, with an explicitly given $C$ depending only on $m$ and the degree and height of $F$.

Mahler \cite{mah} proved a $P$-adic analogue of the Gel'fond--Schneider theorem. Gel'fond \cite{gel3} gave an effective estimate for linear forms in two $P$-adic logarithms which was generalized by Coates \cite{coa1} to arbitrarily many $P$-adic logarithms. Using his estimate, Coates \cite{coa2} proved that $|p,q|\le C'$ for every solution $p,q,z_1,\ldots,z_t$ of \eqref{eq42}, with an explicit bound $C'$ depending only on $t$, the maximum of the primes $P_1,\ldots,P_t$ and the degree and height of $F$.

Baker's and Coates' bounds for linear forms in logarithms and for the solutions of \eqref{eq41} and \eqref{eq42} were later improved by many people; for references see e.g. Baker and W\"ustholz \cite{baw} and Bugeaud \cite{bug}. The best known bounds for \eqref{eq41} and \eqref{eq42}, proved over number fields, are due to Bugeaud and Gy\H ory \cite{bugy} and Gy\H ory and Yu \cite{gyoy}.
In \cite{bugy} it is proved that all solutions $p,q$ of \eqref{eq41} satisfy
$$|p,q|<\exp\{c(n)H^{2n-2}(\log H)^{2n-1}\log M\},$$
where $c(n)=3^{3(n+9)}n^{18(n+1)}$, and $M,H(\ge 3)$ are upper bounds for $|m|$ and for the maximum absolute value of the coefficients of $F$, respectively. A similar upper bound is given in \cite{bugy} for the solutions of \eqref{eq42}, but that bound depends also on $t$ and the maximum of $P_1,\ldots,P_t$. In terms of $c(n)$ and $H$ better bounds are obtained in \cite{gyoy} which depend, however, on some parameters of the splitting field of $F$. We note that the exponential dependence is a consequence of the exponential character of the bounds for linear forms in logarithms. 

Effective generalizations to equations over finitely generated domains can be found in B\'erczes, Evertse and Gy\H ory \cite{begy}.

Mahler's finiteness result \cite{Mahler1933a} concerning equation \eqref{eq42} implies that the greatest prime factor $P(F(p,q))$ of $F(p,q)$ at integral points $(p,q)$ tend to infinity as $|p,q|\rightarrow\infty$. Coates \cite{coa2} deduced from his effective theorem mentioned above the first general effective lower bound for $P(F(p,q))$ by showing that
\begin{align}\label{eq411}
P(F(p,q))\gg(\log_2|p,q|)^{1/4}.
\end{align}
Here and below $\log_i$ denotes the $i$th iterate of the logarithmic function with $\log_1=\log$. The lower estimate \eqref{eq411} was improved by several authors; the best known estimate, due to Gy\H ory and Yu \cite{gyoy}, is
\begin{align}\label{eq421}
P(F(p,q))\gg \log_2 |p,q|\cdot \frac{\log_3 |p,q|}{\log_4 |p,q|}.
\end{align}
In \eqref{eq411} and \eqref{eq421} the constants implied by $\gg$ are effective and depend only on the degree and height of $F$. We note that \eqref{eq421} was established in a more general form, over number fields.

From Coates' explicit upper bound for the solutions of equation \eqref{eq42} one can easily deduce an explicit upper bound for the solutions of the $S$-unit equation \eqref{2.4} and its weighted version \eqref{2.6a} over $\mathbb{Q}$. The first explicit bounds for the solutions of $S$-unit equations over number fields were obtained by Gy\H ory \cite{gy1,gy3}. These bounds were later improved by several people and led to many applications; see e.g. Evertse and Gy\H ory \cite{evgy3} and the references given there. The best known bounds over number fields are due to Gy\H ory and Yu \cite{gyoy}. For effective generalizations to unit equations over finitely generated domains, see Evertse and Gy\H ory \cite{evgy2}.

\section{Higher dimensional generalizations of Thue equations, Thue--Mahler equations and unit equations}\label{sec5}

Let again $S=\{P_1,\ldots,P_t\}$ be a set of $t\ge 0$ primes, and denote by $\mathbb{Z}_S$ the ring of $S$-integers in $\mathbb{Q}$, i.e. those rational numbers whose denominators do not contain any prime factor different from $P_1,\ldots,P_t$. We consider decomposable form equations of the form
\begin{align}\label{eq51}
F(q_1,\ldots,q_n)=m\quad\textrm{in}\quad q_1,\ldots,q_n\in\mathbb{Z}_S,
\end{align}
where $m\in\mathbb{Z}_S\setminus\{0\}$ and $F\in\mathbb{Z}[X_1,\ldots,X_n]$ is a decomposable form, i.e. a homogeneous polynomial which factorizes into linear factors over $\overline{\mathbb{Q}}$.

Decomposable form equations are of basic importance in Diophantine number theory. For $n=2$, equation \eqref{eq51} can be written in the form \eqref{eq42} and if $F$ is irreducible and of degree $\ge 3$, Mahler's result applies. Conversely, equation \eqref{eq42} can be easily reduced to finitely many equations of the shape \eqref{eq51}. For $n\ge 2$, further important classes of decomposable form equations are norm form equations, discriminant form equations and index form equations. For norm form equations over $\mathbb{Q}$ Schmidt \cite{schmidt} (case $t=0$), Schlickewei \cite{sch} (case $t\ge 0$), and over number fields Laurent \cite{lau} obtained finiteness results for equation \eqref{eq51}, thereby considerably generalizing Mahler's finiteness theorem on equation \eqref{eq42}. For discriminant form equations and index form equations Gy\H ory \cite{gy2,gy4} provided finiteness criteria. The proofs in \cite{schmidt}, \cite{sch} and \cite{lau} are ineffective because they depend of Schmidt's Subspace Theorem and its $p$-adic generalization, while the proofs in \cite{gy2}, \cite{gy4} are based on Baker's effective theory of logarithmic forms, hence are effective.

Evertse and Gy\H ory \cite{evgy1} gave a general finiteness criterion for equation \eqref{eq51}. Let $\mathcal{L}$ be a maximal set of pairwise linearly independent linear factors of $F$ over $\overline{\mathbb{Q}}$. A non-zero subspace $V$ of the $\mathbb{Q}$-vector space $\mathbb{Q}^n$ is said to be $\mathcal{L}$-\textit{non-degenerate} or $\mathcal{L}$\textit{-degenerate} according as $\mathcal{L}$ does or does not contain a subset of at least three linear forms which are linearly dependent on $V$, but pairwise linearly independent on $V$. In particular, $V$ is $\mathcal{L}$-degenerate if $V$ has dimension $1$. We call $V$ $\mathcal{L}$\textit{-admissible} if no form in $\mathcal{L}$ is identically zero on $V$. Evertse and Gy\H ory \cite{evgy1} proved that the following two statements are equivalent:
\begin{enumerate}
	\item[\textit{(i)}] \textit{Every $\mathcal{L}$-admissible subspace of $\mathbb{Q}^n$ of dimension $\ge 2$ is $\mathcal{L}$-non degenerate.}
	\item[\textit{(ii)}] \textit{For any finite set $S$ of primes and any non-zero $m\in\mathbb{Z}_S$, equation \eqref{eq51} has only finitely many solutions.}
\end{enumerate}
This was proved in a more general form, over finitely generated domains over $\mathbb{Z}$.

The proof of the above finiteness criterion depends on the following finiteness result on multivariate unit equations of the form
\begin{align}\label{eq52}
a_1 u_1 +\cdots + a_n u_n =1\quad\textrm{in}\quad u_1,\ldots,u_n\in\Gamma,
\end{align}
where $a_1,\ldots,a_n$ are non-zero elements of a number field $K$ and $\Gamma$ is a finitely generated subgroup of $K^\ast$. This equation is a generalization of \eqref{2.6a}. A solution $u_1,\ldots,u_n$ of \eqref{eq52} is called \textit{degenerate} if there is a vanishing subsum on the left hand side of \eqref{eq52}. In this case \eqref{eq52} has infinitely many solutions if $\Gamma$ is infinite. As a considerable generalization of Mahler's finiteness theorem on $S$-unit equations \eqref{2.6a}, van der Poorten and Schlickewei \cite{por} and independently Evertse \cite{eve1} proved that equation \eqref{eq52} has only finitely many non-degenerate solutions. As is pointed out in Evertse and Gy\H ory \cite{evgy1}, this theorem and the implication $(i)\Rightarrow (ii)$ concerning equation \eqref{eq51} are equivalent statements. For further related results, including bounds for the number of solutions, applications and references, see Evertse and Gy\H ory \cite{evgy3}.

\end{document}